\documentclass[conference]{IEEEtran}
\IEEEoverridecommandlockouts

\usepackage{subfigure}
\usepackage{amssymb, amsmath, amsfonts}
\usepackage{empheq}
\usepackage{caption, cite}
\usepackage{graphicx}
\usepackage{optidef}
\usepackage{amsthm}
\usepackage{tikz}
\usepackage{amsmath}
\usetikzlibrary{shadows,arrows,positioning}
\pgfdeclarelayer{background}
\pgfdeclarelayer{foreground}
\pgfsetlayers{background,main,foreground}

\usepackage[linesnumbered,ruled,vlined]{algorithm2e}
\usepackage{booktabs}
\usetikzlibrary{external}
\usepackage{pgfplots}
\usepackage{cases}
\theoremstyle{definition}
\theoremstyle{remark}

\newtheorem{lem}{Lemma}
 \newtheorem{ass}{Assumption}

 \newtheorem{remark}{Remark}
 
\newtheorem{definition}{Definition}
\newtheorem{theorem}{Theorem}
\begin{document}
\title{ \hspace*{\fill} \\\hspace*{\fill} \\ \LARGE \bf{A Communication-Efficient and Differentially-Private\\ Distributed Generalized Nash Equilibrium Seeking Algorithm\\ for Aggregative Games}}
\author{Wenqing Zhao, Antai Xie, Yuchi Wu, Xinlei Yi, and Xiaoqiang Ren
\thanks{W. Zhao, A. Xie, Y. Wu and X. Ren are with the School of Mechatronic Engineering and Automation, Shanghai University, Shanghai, China. Emails: \{wqzhao, xatai, wuyuchi, xqren\}@shu.edu.cn. }
\thanks{X. Yi is with Department of Control Science and Engineering, College of Electronics and Information Engineering, Tongji University, Shanghai, 201804, China. Email: xinleiyi@tongji.edu.cn.}
\thanks{X. Ren is also with Key Laboratory of Marine Intelligent Unmanned Swarm Technology and System, Ministry of Education, Shanghai 200444, China.}
\thanks{The work was supported in part by the National Natural Science Foundation of China under Grant 62273223,  62336005, and 62461160313, the Project of Science and Technology Commission of Shanghai Municipality under Grant 22JC1401401 and the Fundamental Research Funds for the Central Universities under Grant 08002150267.}
}

\maketitle
\begin{abstract}
This paper studies the distributed generalized Nash equilibrium seeking problem for aggregative games with coupling constraints, where each player optimizes its strategy depending on its local cost function and the estimated strategy aggregation. The information transmission in distributed networks may go beyond bandwidth capacity and eventuate communication bottlenecks. Therefore, we propose a novel communication-efficient distributed generalized Nash equilibrium seeking algorithm, in which the communication efficiency is improved by event-triggered communication and information compression methods. The proposed algorithm saves the transmitted rounds and bits of communication simultaneously. Specifically, by developing precise step size conditions, the proposed algorithm ensures provable convergence, and is proven to achieve $(0,\delta)$-differential privacy with a stochastic quantization scheme. In the end, simulation results verify the effectiveness of the proposed algorithm.
 \end{abstract}

\section{Introduction}\label{introduction} 
Aggregative games model diverse applications, including EV charging \cite{7879192} and atomic routing \cite{roughgarden2007routing}, where each player's cost relies on both its strategy and the aggregation of others'. Despite the complexity of integrating game theory, physical interdependencies constrain players' strategies \cite{8851259}. Generalized Nash equilibrium (GNE) emerges as a solution, where players minimize individual costs while satisfying coupling constraints. GNE, a fundamental concept, has been extensively studied in power systems \cite{wang2021distributed}, traffic networks \cite{9749902}, and unmanned systems \cite{dai2024distributed}. 

The proliferation of networked systems has led to a growing emphasis on distributed algorithm design, where players can only interact with their neighbors within the network. To seek GNE, many distributed approaches have been proposed \cite{LIANG2017179,9137646, ZOU2021109535}. In many existing studies, coupling constraints are addressed through primal-dual reformulations. For example, \cite{LIANG2017179} proposed a distributed continuous-time algorithm for nonlinear aggregative games with linear coupling constraints. This was extended by \cite{9137646} to weight-unbalanced directed graphs, while \cite{ZOU2021109535} introduced a distributed observer to estimate opponents’ decisions. Although the primal--dual method features computationally friendly action updates at each iteration step, the distributed algorithms require a considerable amount of network resources. Transmissions of comprehensive data among neighboring players may suffer from communication bottlenecks attributable to coupling constraints. Developing an effective algorithm for identifying GNE in aggregative games continues to pose significant challenges.

To address the need of an efficient information transmission scheme in aggregative games, efforts have been made recently. Note that, event-triggered mechanisms are effectively methods to reduce the rounds of information exchanges between players, which have been extensively studied in \cite{huo2024distributed,ding2024distributed,wang2024adaptive}. For instance, \cite{huo2024distributed} proposed a stochastic event-triggered algorithm for exact Nash equilibrium (NE) seeking, while \cite{ding2024distributed} introduced an event- and switching-activated communication scheme. Under coupling constraints, \cite{wang2024adaptive} developed a dynamic event-triggered mechanism for non-convex nonsmooth aggregative games. These protocols pause information exchange when changes are negligible, saving communication resources. Despite the reduction in the communication rounds, the amount of data transmitted per iteration may still exceed the bandwidth limit. To tackle this challenge, compression mechanisms are developed to save communication transmitted bits. For example, \cite{10531790} developed a compression-based algorithm for NE seeking to save transmitted bits. \cite{chen2023efficient} further used stochastic compressor and event-trigger mechanism simultaneously in the distributed NE seeking algorithm design. Notably, existing distributed algorithms \cite{chen2022linear,chen2023efficient,10551403,10531790} neglect coupling constraints, making efficient GNE seeking in such settings a challenge. To address this, we integrate event-triggered and compression mechanisms to simultaneously reduce transmission rounds and bits.

Besides high communication costs, another challenge exists that there exists a potential for information leakage during direct information transmission between players. Differential privacy (DP) has become a critical tool for providing privacy protection across computer science, control engineering, and communication technology. Multiple types of research have been illustrated for distributed NE seeking with DP method \cite{9416892,10526393, wang2024ensuring, ZENG201920}. A privacy-preserving distributed algorithm for seeking the NE was designed in \cite{9416892} for aggregative games and \cite{10526393} focused on quadratic network games. \cite{wang2024ensuring} extended privacy-preserving mechanisms for NE seeking to directed graph. Considering coupling constraints, \cite{ZENG201920} proposed a GNE seeking strategy for multi-cluster games. Though \cite{wang2022ensure} investigated  GNE seeking with privacy preservation, information exchange efficiency was not considered. Privacy and communication efficiency are seldom explored together in distributed GNE seeking for aggregative games though it is of great interest.

The motivation of this paper is to study aggregative games with coupling constraints and to design an efficient distributed algorithm to seek GNE of the games. The principal contributions of this paper can be outlined as follows:
\begin{enumerate}
\item We provide a novel communication-efficient distributed GNE seeking algorithm (Algorithm 1), combining an event-triggered mechanism to decrease transmission rounds and a stochastic compressor to minimize transmitted bits. 

\item Besides achieving efficient communication, our algorithm ensures convergence to exact GNE by developing precise step size conditions (Theorem 1). In contrast to existing GNE seeking works ensuring accuracy \cite{wang2022ensure}, our proposed algorithm further reduces the communication cost.  

\item We show that $(0,\delta)$-differential privacy is achieved under a stochastic compressor (Theorem 2).  Particularly, the algorithm guarantees asymptotic convergence to the GNE and privacy protection simultaneously.
\end{enumerate}

The organization of this paper is as follows: Section II presents the fundamental concepts and establishes the formulation for the problem. Section III provides an event-triggered and compressed distributed GNE seeking algorithm with an analysis of convergence and privacy. Section IV provides simulation examples to confirm the obtained results, and Section V concludes the paper.

\emph{Notations:}
Let $\mathbb{R}$ ($\mathbb{R}^+$) denote the set of real (positive real) numbers, $\mathbb{Z}$ the integers, and $\mathbb{N}$ the natural numbers. For $n \in \mathbb{N}$, $\mathbb{R}^n$ ($\mathbb{R}^{n \times d}$) represents $n$-dimensional real vectors ($n \times d$ real matrices). Let $P^\mathrm{T}$ and $[P]_{ij}$ denote the transpose and $(i,j)$-th entry of matrix $P$, respectively. The gradient of $f$ at $\mathbf{x}$ is denoted by $\nabla f(\mathbf{x})$, with $\nabla_i f(\mathbf{x})$ as its $i$-th component. Define $\mathbf{1}_N$ ($\mathbf{0}_N$) as the $N$-dimensional all-ones (all-zeros) vector, and $I_d$ as the $d$-dimensional identity matrix. The inner product and Euclidean (induced-$2$) norm are denoted by $\langle \cdot, \cdot \rangle$ and $|\cdot|$, respectively. Let $\Pi_{\mathcal{D}}(\cdot)$ be the Euclidean projection onto a closed convex set $\mathcal{D}$. For probability and expectation, we use $\mathbb{P}(\cdot)$ and $\mathbb{E}[x]$. Finally, $\mathrm{col}(x_1,...,x_m)=[x_1^\mathrm{T},...,x_m^\mathrm{T}]^\mathrm{T}$ stacks vectors $x_1,...,x_m$.

\section{Preliminaries and Problem Formulation}
\subsection{Problem Formulation}
Let us examine an aggregative game comprising $N$ players, where each participant $i\in \mathcal{V}= \{1, 2,...,N\}$ is associated with a decision vector $x_i \in \Omega_i$, with $\Omega_i \subseteq \mathbb{R}^d$ representing a compact convex set, and they share the constraint
\begin{align}
g(\Bar{x}) \leq 0,~~~\Bar{x}=\frac{1}{N}\sum_{i=1}^N x_i \in \Upsilon \subset \mathbb{R}^d
\end{align}
where $g : \mathbb{R}^d \to \mathbb{R}$ is the constrained function. We assume that each player aims at minimizing its local cost function $J_i(x_i,h(\mathbf{x})):\Omega \to \mathbb{R}$, where $\mathbf{x}=\mathrm{col}(x_1,x_2,...,x_N) \in \Omega = \Omega_1 \times \Omega_2 \times ...\times \Omega_N$ and $h(\mathbf{x})=\Bar{x}$ is the aggregation function. The function $J_i(x_i,h(\mathbf{x}))$ is continuously differentiable and convex in $x_i$ for every fixed $h(\mathbf{x})$. In particular, the problem is formulated as follows:

\textit{Problem 1}: Each player $i$ intends to
\begin{mini*}
 {x_i \in \Omega_i} {J_i(x_i,\Bar{x})}{}{}
\addConstraint{g(\Bar{x})\leq 0.}  
\end{mini*}

The resolution of Problem 1 yields a GNE, characterized as follows.
\begin{definition}\cite{deng2021distributed}
 A strategy profile $\mathbf{x}^*=\mathrm{col}(x_1^*,x_2^*,...,x_N^*) \in \Omega \cap \Upsilon$ is called a GNE of the constrained game if
\begin{align}
J_i(x_i^*,h(\mathbf{x}^*)) \leq J_i(x_i,\frac{1}{N}x_i+\sum_{j \neq i} \frac{1}{N}x_j^*), \nonumber
\end{align}
for all $x_i:(x_i,x_{-i}^*)\in \Omega \cap \Upsilon$, where $x_{-i}^*=\mathrm{col}(x_1^*,\dots, x_{i-1}^*, x_{i+1}^*, \dots, x_N^*)$.
\end{definition}

If $J_i(x_i, h(x))$ is regarded as a function of $x$, then $J_i(x_i, h(x))$ can be written as $J_i(\mathbf{u}), u \in \Omega$ for simplism and the following assumptions about the gradient are associated with the GNE.

\begin{ass} \label{lipschitz}
\begin{enumerate}
\item 
 $J_i(\cdot)$ is Lipschitz continuous, i.e., there exists a constant $l_J$ such that for all $\mathbf{u}, \mathbf{v} \in \Omega$ and $i \in \mathcal{V}$, 
$$ \Vert J_i(\mathbf{u})-J_i(\mathbf{v})\Vert \leq l_J\|\mathbf{u}-\mathbf{v}\|.$$
\item $g(\cdot)$ is Lipschitz continuous, i.e., there exists a constant $l_g$ such that for all $u,v \in \mathbb{R}^d$, 
$$\Vert g(u)-g(v)\Vert \leq l_g\|u-v\|. $$
\end{enumerate}
The Lipschitz continuous condition implies $g(\cdot)$ is bounded on the compact set $\mathbb{R}^d$. In other words, there exists a constant $C_g>0$ such that $ \Vert g(\cdot)\Vert  \leq C_g.$
\end{ass}

\begin{ass}
\begin{enumerate}
\item
$\nabla J_i(\cdot)$ is Lipschitz continuous, i.e., there exists a constant $G_J$ such that for all $\mathbf{u}, \mathbf{v} \in \Omega$ and $i \in \mathcal{V}$, $$ \Vert \nabla_iJ_i(\mathbf{u})-\nabla_iJ_i(\mathbf{v})\Vert \leq G_J\|\mathbf{u}-\mathbf{v}\|.$$
\item  $\nabla_ig(\cdot)$ is Lipschitz continuous, i.e., there exists a constant $G_g$ such that for all $u,v \in \mathbb{R}^d$, $$\Vert \nabla_ig(u)-\nabla_ig(v)\Vert \leq G_g\|u-v\|. $$
\end{enumerate}
\end{ass}

Convexity and differentiability assumptions on function $J(\cdot)$, along with Assumption 2, guarantee the uniqueness and existence of NE points. This conclusion has been widely utilized in numerous articles \cite{liu2022distributed,belgioioso2020distributed, wang2024distributed}.

We denote $\lambda_i$ as the Lagrange multiplier of player $i$. The Lagrange function is $\mathcal{L}_i(\mathbf{x},\lambda_i)=J_i(x_i,\Bar{x})+\lambda_ig(\Bar{x})$. To seek GNE, a Karush-Kuhn-Tucker (KKT) condition for the considered game is determined in  \cite{zou2021continuous}, as 
\begin{align}
& x_i-\Pi_{\Omega_i}[x_i-(\nabla_iJ_i(x_i,\Bar{x})+\lambda_i\nabla_i g(\Bar{x}))]=0,\nonumber\\
&\lambda_i \geq 0, ~\lambda_i g(\Bar{x})=0, ~g(\Bar{x}) \leq 0,~i \in \mathcal{V}.\label{kkt}
\end{align}

A classical GNE seeking algorithm leveraging the projection operator is developed, with each player $i \in\mathcal{V}$ updating its local variables as follows: 
\begin{subequations}
\begin{align} 
x_{i,k+1}&=\Pi_{\Omega_i}[x_{i,k}-\gamma_k(\nabla_i J_i(x_{i,k},\Bar{x}_k)+\lambda_{i,k}\nabla_i  g(\Bar{x}_k))], \label{dynamica}\\
    \lambda_{i,k+1}&=\Pi_{\mathbb{R}^+}[\lambda_{i,k}+\gamma_kg(\Bar{x}_k)], \label{dynamicb}
\end{align}
\end{subequations}
where step sizes $\gamma_k>0$, $x_{i,k}$ and $\lambda_{i,k}$ is the decision and dual variable of player $i$ at iteration $k$, respectively.

Next, a lemma characterizing the fixed point $\mathbf{x}^*=\mathrm{col}(x_1^*,x _2^*,...,x_N^*)$ and $\mathbf{\lambda}^* = \mathrm{col}(\lambda_1^*,\lambda_2^*,...,\lambda_N^*)$ of dynamic (4) is presented as follows.
\begin{lem}\cite{cominetti2012modern}
   The fixed point $(\mathbf{x}^*, \mathbf{\lambda}^*)$ of (4) is also the GNE of Problem 1.
\end{lem}

\subsection{Graph Theory}
We analyze $N$ players communicating through an undirected graph $\mathcal{G} = (\mathcal{V}, \mathcal{E})$, where $\mathcal{E} \subseteq \mathcal{V} \times \mathcal{V}$ represents the edges. An edge $(i, j)$ exists between players $i$ and $j$ if they can communicate. A path in the graph is a sequence of edges that connects a sequence of nodes. The graph is connected if there exists a path between every pair of nodes. The Laplacian matrix $A = [a_{ij}] \in \mathbb{R}^{N \times N}$ is defined such that $a_{ij} = 1$ if $(j, i) \in \mathcal{E}$ and $a_{ij} = 0$ otherwise. The diagonal entries are given by $a_{ii} = -\sum_{j \in \mathcal{N}_i} a_{ij}$ for all $i \in \mathcal{V}$, where $\mathcal{N}_i$ denotes the set of neighbors of node $i$. We assume the following assumption.
\begin{ass}
    The undirected graph $\mathcal{G}$ is connected. 
\end{ass}

The Assumption 3 implies the matrix $A = \{a_{ij} \} \in \mathbb{R}^{N\times N}$ is symmetric and satisfies $\mathbf{1}_N^\mathrm{T}A=\mathbf{0}_N^\mathrm{T}$, $A\mathbf{1}_N=\mathbf{0}_N$, and $\Vert I_N+A-\frac{\mathbf{1}_N\mathbf{1}_N^\mathrm{T}}{N}\Vert <1$.

\subsection{Compression Method}
To enhance communication efficiency, we employ a stochastic compressor to encode the original state information before transmission, satisfying the following assumption: 

\begin{ass}
For some constant $\sigma$ and any $x \in \mathbb{R}^d$, $\mathbb{E}[\mathcal{C}(x)|x] = x$ and $\mathbb{E}[\Vert \mathcal{C}(x) - x\Vert ^2] \leq \sigma^2$. The stochastic mechanism in each player’s compression is statistically independent. \label{compress}
\end{ass}
\begin{remark}
Compressors under Assumption \ref{compress} are commonly used in distributed algorithms, see, e.g., \cite{9224135,wang2022quantization,10531790}. 
\end{remark}
\subsection{Event-triggered Mechanism}
 We assume the following to hold at the start of iteration $k$, player $i$ decides whether the local data is worthy of transmission by an event-triggered condition as:
\begin{align}
\mathbb{I}_{i,k}=\left\{
\begin{aligned}
&1,~\text{if}~~\Vert v_{i,k}-\Tilde{v}_{i,k-1}\Vert \geq \tau_{i,k}, \\
&0,~\text{else},\\
\end{aligned}
\right.\label{eventtrigger}
\end{align} 
where $v_{i,k}$ is the current compressed value, $\Tilde{v}_{i,k-1}$ is the latest sent value, $\mathbb{I}_{i,k}$ is the triggering indicator and $\tau_{i,k}$ is the triggering thresholds. In this scenario, data exchange is initiated only if the deviation between the current compressed state and the last transmitted state meets or exceeds $\tau_{i,k}$. The triggering threshold is subject to the following assumption.
\begin{ass}
The thresholds $\{\tau_{i,k}\}_{k \geq 0}, i \in \mathcal{V}$ satisfies
$$\tau_{i,k} \geq 0, ~ \lim_{k \to \infty} \tau_{i,k} =0. $$\label{thresholds}
\end{ass}\begin{remark}
The decaying thresholds $\{\tau_{i,k}\}$ must ensure infinite-horizon information sharing. Without this, the diminishing difference between consecutive estimates would prevent event triggering after finite time. Assumption \ref{thresholds}, satisfied by common sequences (e.g., \cite{chen2023efficient,he2022event}), allows thresholds of the form $\tau_{i,k} = B_i \alpha_i^k$, where $B_i > 0$ and $0 < \alpha_i < 1$.
\end{remark}

\section{Main Results}
\subsection{Algorithm Design}
In this section, we develop an efficient GNE seeking strategy using the compressor and event-trigger mechanism based on dynamic (4) from Section II. At iteration step $k$, every player $i\in \mathcal{V}$ updates their local information consistent with 
\begin{subequations}
\begin{align}
x_{i,k+1}&=\Pi_{\Omega_i}[x_{i,k}-\gamma_k(\nabla_i J_i(x_{i,k},y_{i,k})+\lambda_{i,k}\nabla_i g(y_{i,k}))],\label{update1a}\\
\lambda_{i,k+1}&=\Pi_{\Lambda_i}[\lambda_{i,k}+\gamma_kg(y_{i,k})],\label{update1b}
\end{align}
\end{subequations}
where $\gamma_k>0$ is the step size, $\nabla_i J_i(x_i,y_i)=\{\frac{\partial J_i(x_i,\Bar{x})}{\partial x_i}+\frac{\partial J_i(x_i,\Bar{x})}{\partial \Bar{x}}\frac{\partial \Bar{x}}{\partial x_i} \}_{\Bar{x}=y_i}$ and the bound on $\lambda_i$ is formally established in \cite{nedic2009approximate} as $\Lambda_i=\{0 \leq ~ \lambda_i \leq \frac{J_i(\hat{\mathbf{x}})-\min_{x_i \in \Omega_i}\mathcal{L}_i(\mathbf{x},\lambda_i')}{\min\{-g(x')\}}=\Bar{\Lambda}_i\}$ for an arbitrary $\lambda_i'\geq 0$, $\hat{\mathbf{x}}=(\hat{x}_1,\dots, \hat{x}_N)$ is a Slater point of Problem 1 and $x'=\frac{1}{N}\sum_{i=1}^N \hat{x}_i$. In addition, $\hat{\mathbf{x}}$ can be either locally computed based on \cite{zhu2011distributed} or can be taken as a superset estimate as described in \cite{belgioioso2020distributed}. And $y_i$  is the intermediate variable which is defined as:
\begin{align}
 y_{i,k+1}=y_{i,k}+\eta_k\sum_{j\in \mathcal{N}_i}^N a_{ij}(\Tilde{v}_{j,k}-v_{i,k})+x_{i,k+1}-x_{i,k},\label{update2}
\end{align}
where $\eta_k>0$ is the step size,  $v_{i,k}=\mathcal{C}(y_{i,k})$ is the compressed version of $y_{i,k}$ and $\Tilde{v}_{j,k}$ is the latest triggered output from player $j$. To improve efficiency, the compressed value $v_{i,k}$ is exchanged by event-triggered condition (5). The detailed procedure of this algorithm is outlined in Algorithm 1.
\begin{algorithm}
    \SetAlgoLined
    \caption{Event-triggered and Compressed Distributed GNE Seeking Algorithm}
    \KwIn{Stopping time $K$, step sizes $\eta_k, \gamma_k$, thresholds $\tau_{i,k}$ for $i \in \mathcal{V}$}
	\textbf{Initialize:} $\mathbf{x}_0 \in \mathbb{R}^{Nd}$, $\mathbf{y}_0 = \mathbf{x}_0$.\\
 \For{each agent $i\in \mathcal{V}$}{
      Send $v_{i,0}=\mathcal{C}(y_{i,0})$ to neighbors and set $\Tilde{v}_{i,0}=v_{i,0}$.\\
    \For{$k=1,...,K$}{
    Take $v_{i,k}=\mathcal{C}(y_{i,k})$. \\ 
     \eIf{$\Vert v_{i,k}-\Tilde{v}_{i,k-1}\Vert  \geq \tau_{i,k}$}{Send $v_{i,k}$ to neighbors and set $\Tilde{v}_{i,k}=v_{i,k}$. }{Set $\Tilde{v}_{i,k}=\Tilde{v}_{i,k-1}$.}
     And do local updates:
      \begin{align}
x_{i,k+1}&=\Pi_{\Omega_i}[x_{i,k}-\gamma_k(\nabla_i J_i(x_{i,k},y_{i,k})\nonumber\\
&~~+\lambda_{i,k}\nabla_i g(y_{i,k}))],\nonumber\\
\lambda_{i,k+1}&=\Pi_{\Lambda_i}[\lambda_{i,k}+\gamma_kg(y_{i,k})]. \nonumber
\end{align}\\
    Update auxiliary variable:
    $$y_{i,k+1}=y_{i,k}+\eta_k\sum_{j\in \mathcal{N}_i}^N a_{ij}(\Tilde{v}_{j,k}-v_{i,k})+x_{i,k+1}-x_{i,k}.$$
	}
 }
 \textbf{Output:} $\mathbf{x}_{K+1}$
\end{algorithm}
\subsection{Convergence Analysis}
In this section, we present a rigorous analysis that demonstrates that Algorithm 1 drives all players’ decisions to $\mathbf{x}^*$, contingent on appropriately chosen step sizes.

We introduce the following assumption and lemmas to establish the main results.
\begin{ass}
    For $i\in \mathcal{V}$ and $k\in \mathbb{Z}^+$, the thresholds $\{\tau_{i,k}\}_{k \geq 0}$ satisfies
    \begin{align}
\sum_{k=0}^\infty \eta_k\max_{i \in \mathcal{V}}(\tau_{i,k})^2<\infty.
    \end{align}
   \label{weaktao}
\end{ass}

\begin{lem}
    (Lemma 4 of\cite{wang2024robust}) Under Assumption 3, for a positive sequence $\{\eta_k\}$ satisfying $\sum_{k=0}^\infty \eta_k^2<\infty$, there exists a $T\geq 0$ such that $\Vert I_N+\eta_kA-\frac{\mathbf{1}_N\mathbf{1}_N^\mathrm{T}}{N}\Vert <1-\eta_k|\rho_2|$ holds for all $k\geq T$, where $\rho_2$ is $A$'s second largest eigenvalue in terms of their real parts. \label{rhoeigenvalue}
\end{lem}

\begin{lem}\cite{wang2023tailoring}
    Let $\{p_k\}$,$\{a_k\}$, and $\{b_k\}$ be random nonnegative scalar sequences, and $\{q_k\}$ be a deterministic nonnegative scalar sequence satisfying $ \sum_{k=0}^\infty a_k <\infty$ a.s., $\sum_{k=0}^\infty q_k =\infty$, $\sum_{k=0}^\infty b_k <\infty$ a.s., and 
    \begin{align}
\mathbb{E}[p_{k+1}|\mathcal{F}_k]\leq (1+a_k-q_k)p_k+b_k, \forall k \geq 0, a.s.
    \end{align}
    where $\mathcal{F}_k$ is the $\sigma$-algebra generated by $\{p_l,a_l,p_l,0\leq l\leq k\}$. Then $\lim_{k \to \infty}p_k =0$ and $\sum_{k=0}^\infty q_kp_k<\infty$ hold almost surely. \label{sequence}
\end{lem}

\begin{lem}\cite{wang2023tailoring}
    Let $\{p_k\}$,$\{u_k\}$, $\{a_k\}$ and $\{b_k\}$ be random nonnegative scalar sequences satisfying $ \sum_{k=0}^\infty a_k <\infty$ a.s., $\sum_{k=0}^\infty b_k <\infty$ a.s., and 
    \begin{align}
 \mathbb{E}[p_{k+1}|\mathcal{F}_k]\leq (1+a_k)p_k-u_k+b_k, \forall k \geq 0, a.s.
    \end{align}
    where $\mathcal{F}_k$ is the $\sigma$-algebra generated by $\{p_l,u_l,a_l,b_l,0\leq l\leq k\}$. Then, $\sum_{k=0}^\infty u_k<\infty$ and $\lim_{k \to \infty}p_k =v$ hold almost surely for a random variable $v \geq 0$. \label{sequence2}
\end{lem}

\begin{lem}
    Under Assumptions 1--6, if $\sum_{k=0}^\infty \eta_k=\infty$, $\sum_{k=0}^\infty \eta_k^2 <\infty$, $\sum_{k=0}^\infty \frac{\gamma_k^2}{\eta_k}<\infty$, then for each player $i \in \mathcal{V}$, $\lim_{k\to \infty}\Vert y_{i,k}-\Bar{x}_k\Vert ^2 =0$ holds almost surely. \label{ybarx}
\end{lem}
\begin{proof}
The proof is provided in Appendix A.
\end{proof}

Based on the preliminary results above, we can show that the primal--dual iteration pairs converge to GNE of Problem~1. 

\begin{theorem}
    Under Assumptions 1--\ref{weaktao}, with the step sizes satisfying $\sum_{k=0}^\infty \eta_k=\infty$, $\sum_{k=0}^\infty \eta_k^2 <\infty$, $\sum_{k=0}^\infty \frac{\gamma_k^2}{\eta_k}<\infty$, the sequence $\{(\mathbf{x}_k,\mathbf{\lambda}_k)\}$ generated by Algorithm 1 converges to the fixed point $(\mathbf{x}^*, \mathbf{\lambda}^*)$ almost surely, namely,  $\lim_{k \to \infty}\Vert x_{i,k}-x_i^*\Vert ^2=0$ and $\lim_{k \to \infty}\Vert \lambda_{i,k}-\lambda_i^*\Vert ^2=0$ a. s. for all $i \in \mathcal{V}$.
\end{theorem}
\begin{proof}
The proof is provided in Appendix B.
\end{proof}

\subsection{Differential Privacy Analysis}
This section presents a rigorous analysis of differential privacy. We introduce the following definitions pertinent to differential privacy in the context of distributed aggregative games.

\begin{definition}\cite{9416892}
 (Adjacent Function Sets): Two function sets $\mathcal{F} ^{(1)} =\{f^{(1)}\}^N_{i=1}$ and $\mathcal{F}^{(2)} = \{f^{(2)}\}^N_{i=1}$ are said to be adjacent if there exists some $i_0 \in \mathcal{V}$ such that $f^{(1)} = f^{(2)}, \forall i \neq i_0$, and $f^{(1)}_i \neq f^{(2)}_i$ for $i=i_0$.
\end{definition}

\begin{definition}\cite{meiser2018approximate}
($(\epsilon, \delta)$-Differentially Privacy): Let $\mathcal{R}$ be a distributed GNE seeking strategy. The strategy $\mathcal{R}$ preserves $(\epsilon, \delta)$-differential privacy if for any pair of adjacent objective function sets $\mathcal{F}^{(1)}$, $\mathcal{F}^{(2)}$ and any observation $\mathcal{O} \subset Range(\mathcal{R})$, the following inequality holds
\begin{align}
 \mathbb{P}\{\mathcal{R}(\mathcal{F}^{(1)}) \in \mathcal{O}\} \leq e^\epsilon \mathbb{P}\{\mathcal{R}(\mathcal{F}^{(2)}) \in \mathcal{O}\}+\delta, 
\end{align}
where $\epsilon \geq 0,~0\leq\delta\leq 1$ and $Range(\mathcal{R})$ is the domain of the observation under $\mathcal{R}$. \label{dp}
\end{definition}

To achieve privacy protection, we introduce a compressor that satisfies Assumption \ref{compress}:

\begin{definition}\cite{10531790}
The element-wise stochastic compressor quantizes a vector $x = \mathrm{col}(x_{(1)}, x_{(2)},..., x_{(d)}) \in \mathbb{R}^d$ as $\mathcal{C}(x) =
[\mathcal{C}(x_{(1)}),\mathcal{C}(x_{(2)}),...,\mathcal{C}(x_{(d)})]^\mathrm{T}$ to the range by a scale factor $\theta \in \mathbb{N}^+$ and $b$ bits, where $2^b\theta>|x_{(i)}|,  i = 1,2,...,d$. For any $l_c\theta \leq x_{(i)} < (l_c + 1)\theta$, the compressor outputs
\begin{align}
\mathcal{C}(x_{(i)})=\left\{
\begin{aligned}
&l_c\theta,~\text{with probability}~~ 1+l_c-x_{(i)}/\theta, \\
&(l_c+1)\theta,~\text{with probability}~~ x_{(i)}/\theta-l_c.\\
\end{aligned}
\right.
\end{align}
\end{definition}
For adjacent function sets $\mathcal{F}^{(1)}$ and $\mathcal{F}^{(2)}$, the eavesdropper is presumed to know the initial states of the algorithm. Thus, $\mathbf{x}_0^{(1)}=\mathbf{x}_0^{(2)}$ and $\mathbf{y}_0^{(1)} = \mathbf{y}_0^{(2)}$ based on the same observation. We indicate that $y_{i_0}^{(1)} \neq y_{i_0}^{(2)} $ for $i=i_0$ and $y_{i}^{(1)} = y_{i}^{(2)} $ for $i\neq i_0$. In the rest, we only have to consider the situation when $i=i_0$. We prove Algorithm~1 can achieve $(0,\delta)$-differential privacy with the compressor from Definition~4. 
\begin{theorem}
     Under Assumptions 1--6 and the conditions of Theorem 1, assume $\Delta y_{i_0,k}= y_{i_0,k}^{(1)}-y_{i_0,k}^{(2)} \leq 2\theta$, and let $\eta_k =\frac{1}{k^s}$ and $\gamma_k=\frac{1}{k^t}$ with  $0.5<s \leq t\leq 1$ and $2t-s>1$, then Algorithm~1 ensures $(0,\delta_k)$-differential privacy where
    \begin{align}
  \delta_k= \min\{ 1, \frac{2t\ln(k)\sqrt{N}}{\theta}(l_J+l_g \max \{\Bar{\Lambda}_i\}) \}. \label{deltak}
    \end{align}
\end{theorem}
\begin{proof}
The proof is provided in Appendix C.
\end{proof}
\section{Simulation}
This section presents numerical simulations to validate the theoretical results. We consider a system of $N$ electricity users, each with specific energy consumption demands, communicating on a randomly generated ring graph \cite{liang2017distributed,grammatico2017dynamic}.  For user $i$, $x_i \in [\underline{r}_i,\overline{r}_i]$ is the energy consumption and $J_i(x_i, \Bar{x})$ is the cost function established by
\begin{align}
J_i(x_i, \Bar{x}) = (x_i - s_i)^2 + p_0 (N*\Bar{x}+p_1)x_i, 
\end{align}
where $s_i$ is  the nominal value of energy consumption and $p_0 (N*\Bar{x}+p_1)x_i$ represents the price. We choose $N=5$,   $s_1=56,s_2=60,s_3=42,s_4=57,s_5=54$, $p_0=0.05$, $p_1=9$ and $[\underline{r}_i,\overline{r}_i]=[30,50]$ and further impose the linear constraint $\sum_{i=1}^N x_i \leq 200$. In this case, the GNE is at
$\mathbf{x}^*=[42.5137,~45.1596,~30,~44.4407,~37.8858]^\mathrm{T}$ through a centralized calculation.

To assess the efficacy of Algorithm 1 we set $\gamma_k=\frac{1}{k^{0.9}}$ and $\eta=\frac{1}{k^{0.7}}$ with the threshold $\tau_{i,k} = 20\times0.8^k$ and $\mathbf{x}_0$ is randomly chosen from $[30, 50]^N$. Due to the randomness of the compressor, we performed simulation runs $100$ to derive the empirical mean of the residual $\frac{1}{T} \sum_1^T\Vert \mathbf{x}_k-\mathbf{x}^*\Vert ^2$ with $T=100$. 

\begin{table}[h]\scriptsize
    \centering
    \caption{Different compression parameters}
\begin{tabular}{cccc}
\toprule

\small Compressors & \small $\theta$ & \small bits &$\small \sigma$     \\
\midrule
\small Compressor 1 & \small 5 &\small 4 &  \small  2.5  \\
\small Compressor 2 & \small 10 &\small 3 & \small  5   \\
\small Compressor 3 & \small 15 &\small 2 &  \small 7.5  \\
\bottomrule
\end{tabular}
\end{table}

\begin{figure}
\begin{center}
\includegraphics[width=3.3in,height=2.5in]{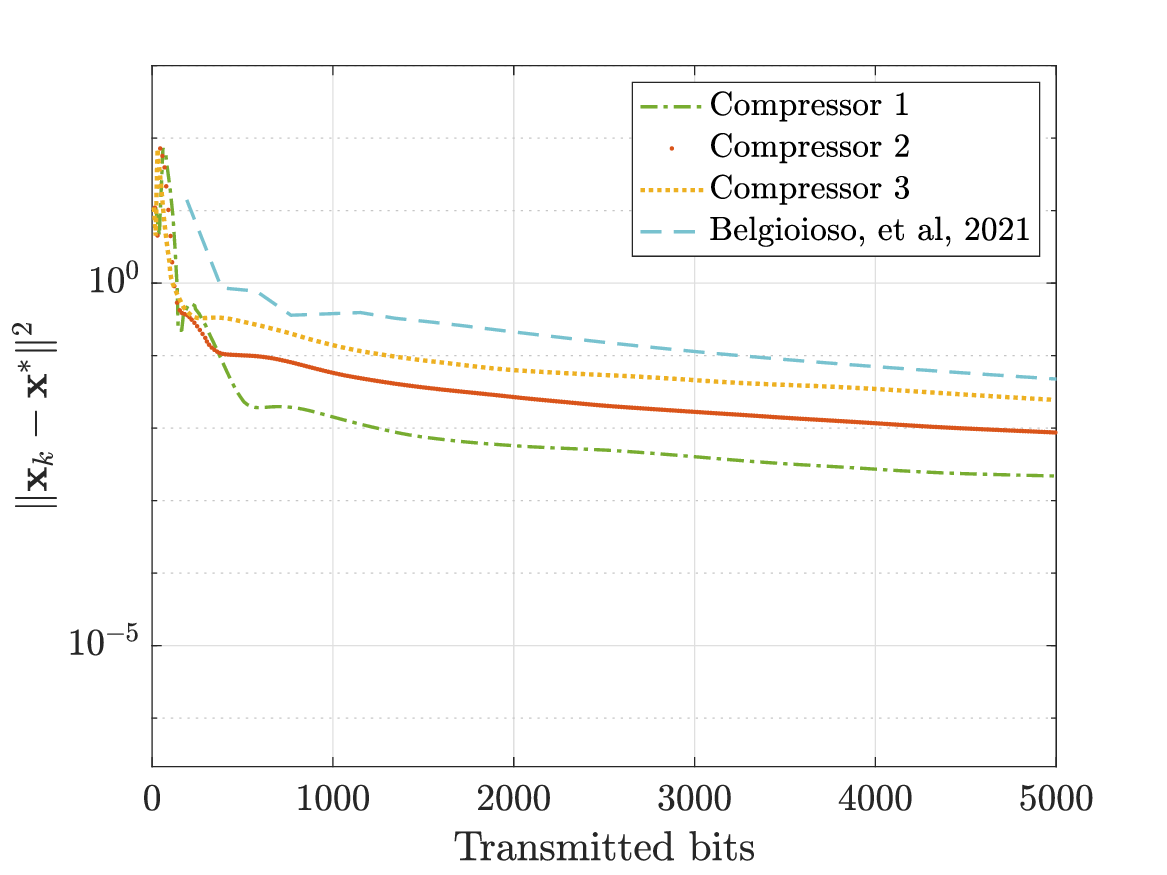}
\caption{Comparison between the proposed Algorithm 1 and the existing GNE approach in \cite{9130079} with various parameter combinations under different parameters in Table I.}
\end{center}
\end{figure}

\begin{figure}
\begin{center}
\includegraphics[width=3.3in,height=2.5in]{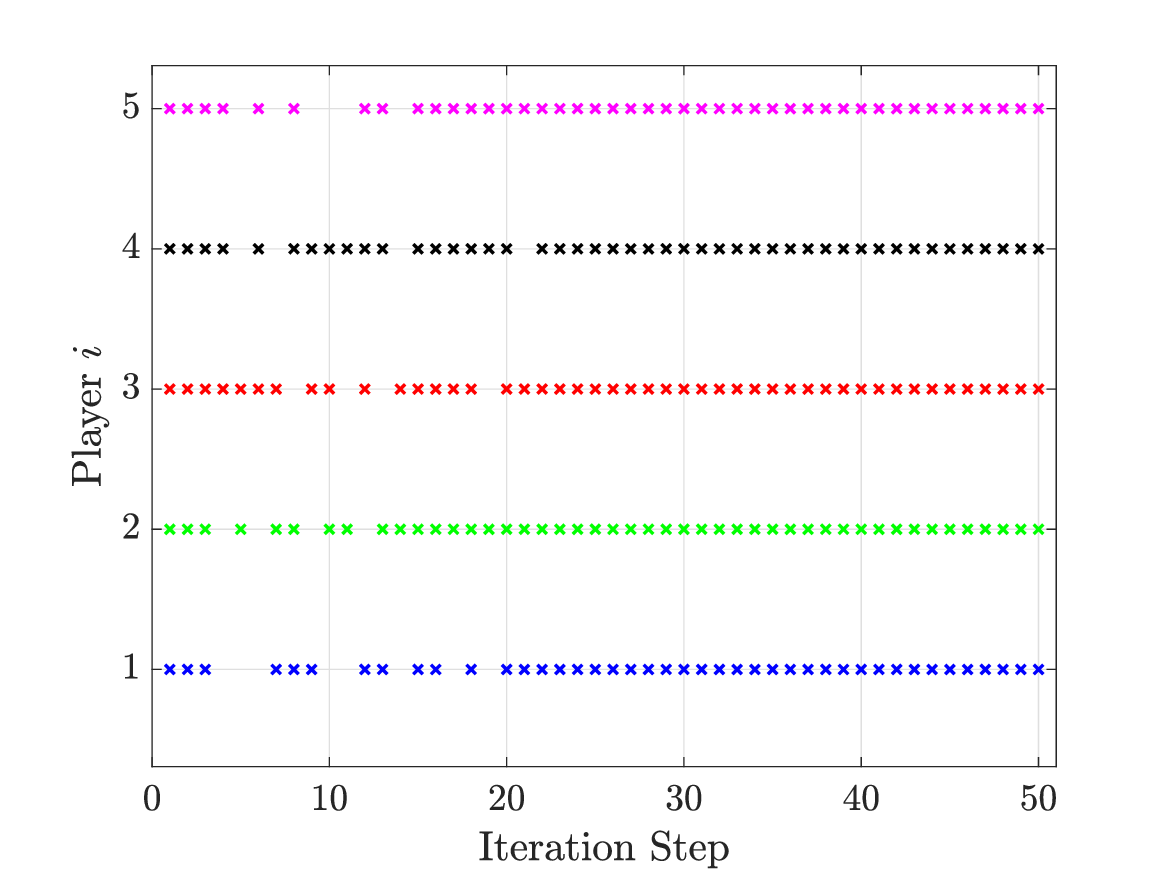}
\caption{Triggering instants for Algorithm~1 with Compressor~3.}
\end{center}
\end{figure}

For comparison, we set different parameters for compressor $\mathcal{C}$, as detailed in Table I. Fig. 1 demonstrates the convergence performance comparison of the proposed Algorithm 1 and the existing GNE-seeking method from \cite{9130079} under different parametric conditions.  It's shown that $\mathbf{x}_k$ converges to the equilibrium point $\mathbf{x}^*$ under Algorithm 1 with different compressors and the Compressor 1 has superior convergence performance while requires fewer bits achieving GNE. The triggering instants of each player are shown in Fig. 2.

\section{Conclusion}
In this paper, we investigated distributed GNE seeking in aggregative games with coupling constraints. We introduced a novel efficient distributed GNE seeking algorithm to save communication costs, which is realized by incorporating an event-trigger mechanism and the stochastic compressor. By developing precise step size conditions, we have shown that the proposed algorithm ensures convergence to an exact NE. In particular, communication rounds are saved by the event-trigger mechanism, and transmitted bits are reduced by stochastic compressors. Privacy-preserving can be further achieved with a stochastic compression scheme with accurate convergence. Numerical simulations validate the efficacy of the algorithm. Future works may include distributed GNE seeking with nonconvex constraints and multiple GNE selecting.

\section{Appendix}
\subsection{Proof of Lemma 5}
Let $e_{i,k}=\Tilde{v}_{i,k}-v_{i,k}$ be the event error and $r_{i,k}=\mathcal{C}(y_{i,k})-y_{i,k}$ be the compress error. Define $A_k=(I_N+\eta_kA-\frac{\mathbf{1}_N\mathbf{1}_N^\mathrm{T}}{N})\otimes I_d$, $R_k=\mathrm{col}(r_{1,k}, r_{2,k},...,r_{N,k})$ and $E_k=[E_{1,k}, ..., E_{N,k}]^\mathrm{T}$, where $E_{i,k}=\sum_{j \in  \mathcal{N}_i}^N a_{ij}e_{j,k}$. 

According to the update rule (\ref{update2}), there is
\begin{align}
y_{i,k+1}=&~y_{i,k}+\eta_k\sum_{j\in \mathcal{N}_i}^N a_{ij}(\Tilde{v}_{j,k}-v_{i,k})+x_{i,k+1}-x_{i,k}\nonumber\\
    =&~y_{i,k}+\eta_k\sum_{j\in \mathcal{N}_i}^N a_{ij}(y_{j,k}-y_{i,k})+\eta_k\sum_{j\in \mathcal{N}_i}^N a_{ij}e_{j,k}\nonumber\\
    &+\eta_k\sum_{j\in \mathcal{N}_i}^N a_{ij}(r_{j,k}-r_{i,k})+x_{i,k+1}-x_{i,k}.\nonumber
\end{align}

Since $\Bar{y}_k=\frac{\mathbf{1}_N^\mathrm{T}\otimes I_d}{N}\mathbf{y}_k$  and $\Bar{y}_{k+1} = \Bar{y}_k + \Bar{x}_{k+1} - \Bar{x}_k$ holds, then $\Bar{y}_k-\Bar{x}_k=\Bar{y}_{k-1} -\Bar{x}_{k-1}=...= \Bar{y}_0 -\Bar{x}_0$. As long as we initialize $\mathbf{y}_0=\mathbf{x}_0$, it evolves that $\Bar{y}_k=\Bar{x}_k$. Given $\sum_{i=1}^N\Vert y_{i,k+1}-\Bar{x}_k\Vert ^2=\Vert \mathbf{y}_{k+1}-\mathbf{1}_N\otimes\Bar{x}_k\Vert ^2$, where $\mathbf{y}_{k}=\mathrm{col}(y_{1,k}, ...,y_{N,k})$, $\mathbf{x}_{k}=\mathrm{col}(x_{1,k},..., x_{N,k})$, $\Bar{x}_k=\frac{\mathbf{1}_N^\mathrm{T}\otimes I_d}{N} \mathbf{x}_k$, we obtain
\begin{align}
    \Vert \mathbf{y}_{k+1}-&\mathbf{1}_N\otimes\Bar{x}_k\Vert ^2\nonumber\\
    \leq &~\Vert ((I_N+\eta_kA)\otimes I_d)\mathbf{y}_k+\eta_k E_k \nonumber\\
    &+(\eta_kA\otimes I_d)R_k+\mathbf{x}_{k+1}-\mathbf{x}_k-\mathbf{1}_N\otimes \Bar{x}_k\Vert ^2\nonumber\\
    \leq &~\Vert A_k\mathbf{y}_k+\eta_kE_k+(\eta_kA \otimes I_d)R_k+\mathbf{x}_{k+1}-\mathbf{x}_k\Vert ^2\nonumber.
\end{align}

As long as $A_k(\mathbf{1}_N \otimes \Bar{x}_k)=\mathbf{0}_{Nd}$ holds, there is $A_k\mathbf{y}_k=A_k(\mathbf{y}_k-\mathbf{1}_N\otimes\Bar{x}_k)$. Thus
\begin{align}
    \Vert &\mathbf{y}_{k+1}-\mathbf{1}_N\otimes\Bar{x}_k\Vert ^2\nonumber\\
    \leq &~\Vert A_k(\mathbf{y}_k-\mathbf{1}_N\otimes\Bar{x}_k)+\mathbf{x}_{k+1}-\mathbf{x}_k+\eta_kE_k\Vert ^2 \nonumber\\
    &+\Vert (\eta_k A\otimes I_d)R_k\Vert ^2+2\langle A_k(\mathbf{y}_k-\mathbf{1}_N\otimes\Bar{x}_k)+\mathbf{x}_{k+1}\nonumber\\
    &-\mathbf{x}_k+\eta_kE_k,(\eta_k A\otimes I_d)R_k\rangle. \label{align1}
\end{align}

Applying the Cauchy–Schwarz inequality, we have
\begin{align}
    &\Vert A_k(\mathbf{y}_k-\mathbf{1}_N\otimes\Bar{x}_k)+\mathbf{x}_{k+1}-\mathbf{x}_k+\eta_kE_k\Vert ^2\nonumber\\
    &\leq  (1+\varepsilon)\Vert  A_k(\mathbf{y}_k-\mathbf{1}_N\otimes\Bar{x}_k)\Vert ^2  \nonumber\\
    &~~+(1+\varepsilon^{-1})\Vert \mathbf{x}_{k+1}-\mathbf{x}_k+\eta_kE_k\Vert ^2. \nonumber
\end{align}

Using Lemma \ref{rhoeigenvalue}, $\exists ~T\geq 0, k \geq T$ such that 
\begin{align}
    \Vert  A_k(\mathbf{y}_k-\mathbf{1}_N\otimes\Bar{x}_k)\Vert ^2 &\leq [(1-\eta_k |\rho_2|)\Vert \mathbf{y}_k-\mathbf{1}_N\otimes\Bar{x}_k\Vert ]^2\nonumber\\
    &\leq (1-\eta_k |\rho_2|)^2  \Vert \mathbf{y}_k-\mathbf{1}_N\otimes\Bar{x}_k\Vert ^2. \nonumber
\end{align}

Given that $\varepsilon=\frac{\eta_k |\rho_2|}{1-\eta_k|\rho_2|}$, it holds that 
\begin{align}
    \Vert A_k(\mathbf{y}_k&-\mathbf{1}_N\otimes\Bar{x}_k)+\mathbf{x}_{k+1}-\mathbf{x}_k+\eta_kE_k\Vert ^2\nonumber\\
    &\leq (1-\eta_k |\rho_2|) \Vert \mathbf{y}_k-\mathbf{1}_N\otimes\Bar{x}_k\Vert ^2\nonumber\\
    &~~+\frac{1}{\eta_k|\rho_2|} \Vert \mathbf{x}_{k+1}-\mathbf{x}_k
    +\eta_kE_k\Vert ^2 . \label{align2}
\end{align}

For $\Vert \mathbf{x}_{k+1}-\mathbf{x}_k\Vert ^2$, according to Assumption 1 and the update rule (\ref{update1a}), there has 
\begin{align}
    \Vert \mathbf{x}_{k+1}&-\mathbf{x}_k\Vert ^2 \nonumber\\
    &=\sum_{i=1}^N \Vert x_{i,k+1}-x_{i,k}\Vert ^2\nonumber\\
     &=\sum_{i=1}^N \Vert -\gamma_k(\nabla_i J_i(x_{i,k},y_{i,k})+\nabla_ig(y_{i,k})\lambda_{i,k}) \Vert ^2\nonumber\\
     &\leq  N\gamma_k^2 (l_J+l_g\max_i\{ \Bar{\Lambda}_i\})^2. \label{align3}
\end{align}

Combine (\ref{align1}), (\ref{align2}) and (\ref{align3}), we can obtain that
\begin{align}
  &\Vert \mathbf{y}_{k+1}-\mathbf{1}_N\otimes\Bar{x}_k\Vert ^2\nonumber\\
  &\leq  ~(1-\eta_k |\rho_2|) \Vert \mathbf{y}_k-\mathbf{1}_N\otimes\Bar{x}_k\Vert ^2\nonumber\\
  &+\frac{1}{\eta_k|\rho_2|} \Vert \mathbf{x}_{k+1}-\mathbf{x}_k+\eta_kE_k\Vert ^2 +\Vert (\eta_k A\otimes I_d)R_k\Vert ^2\nonumber\\
  &+2\left \langle A_k(\mathbf{y}_k-\mathbf{1}_N\otimes\Bar{x}_k)+\mathbf{x}_{k+1}-\mathbf{x}_k+\eta_kE_k,(\eta_k A\otimes I_d)R_k\right \rangle\nonumber\\
  &\leq (1-\eta_k |\rho_2|) \Vert\mathbf{y}_k-\mathbf{1}_N\otimes\Bar{x}_k\Vert ^2+\frac{\gamma_k^2 N}{\eta_k|\rho_2|}(l_J+l_g\max_i\{ \Bar{\Lambda}_i\})^2\nonumber\\
  &+\frac{\eta_k}{|\rho_2|}E_k^2 + \frac{2}{\eta_k|\rho_2|}\langle \mathbf{x}_{k+1}-\mathbf{x}_k,\eta_kE_k\rangle +\Vert (\eta_k A\otimes I_d)R_k\Vert ^2\nonumber\\
  &+2\left \langle A_k(\mathbf{y}_k-\mathbf{1}_N\otimes\Bar{x}_k)+\mathbf{x}_{k+1}-\mathbf{x}_k+\eta_kE_k,(\eta_k A\otimes I_d)R_k\right \rangle. \nonumber
\end{align}

By Algorithm 1, if $\Vert v_{i,k+1}-\Tilde{v}_{i,k}\Vert \geq \tau_{i,k+1}$, then $\Tilde{v}_{i,k+1}=v_{i,k+1}$. Then we get $\Vert v_{i,k+1}-\Tilde{v}_{i,k+1}\Vert =0 < \tau_{i,k+1}$. Else when $\Vert v_{i,k+1}-\Tilde{v}_{i,k}\Vert < \tau_{i,k+1}$, there is $\Tilde{v}_{i,k+1}=\Tilde{v}_{i,k}$. We can also get $\Vert v_{i,k+1}-\Tilde{v}_{i,k+1}\Vert < \tau_{i,k+1}$. So we get $e_{i,k} \leq \tau_{i,k}$. Then $\Vert E_k\Vert ^2\leq N\sum_{i=1}^N\sum_{j \neq i}a_{ij}^2\Vert e_{j,k}\Vert ^2 \leq N\sum_{i=1}^N\sum_{j \neq i}a_{ij}^2\Vert \tau_{j,k}\Vert ^2$. Taking conditional expectation with respect to $\mathcal{F}_k=\{\mathbf{x}_0,\mathbf{x}_1,...,\mathbf{x}_N\}$ and under Assumption 3--4, we obtain that
\begin{align}
    &\mathbb{E}[\Vert \mathbf{y}_{k+1}-\mathbf{1}_N\otimes\Bar{x}_k\Vert ^2|\mathcal{F}_k]\nonumber\\
    \leq &~(1-\eta_k |\rho_2|) \Vert \mathbf{y}_k-\mathbf{1}_N\otimes\Bar{x}_k\Vert ^2+\frac{\gamma_k^2}{\eta_k|\rho_2|}N(l_J+l_g\max_i\{ \Bar{\Lambda}_i\})^2\nonumber\\
  &+\eta_k ^2\mathbb{E}[\Vert (A\otimes I_d)R_k\Vert ^2|\mathcal{F}_k]+\frac{\eta_k}{|\rho_2|}N\sum_{i=1}^N\sum_{j \neq i}a_{ij}^2\tau_{j,k}^2. 
\end{align}

Let $q_k=\eta_k |\rho_2|$, $p_k=\Vert \mathbf{y}_k-\mathbf{1}_N\otimes\Bar{x}_k\Vert ^2$ and $b_k=\frac{\gamma_k^2}{\eta_k|\rho_2|}N(l_J+l_g\max_i\{ \Bar{\Lambda}_i\})^2+\eta_k ^2\mathbb{E}[\Vert (A\otimes I_d)R_k\Vert ^2|\mathcal{F}_k]+\frac{\eta_k}{|\rho_2|}N\sum_{i=1}^N\sum_{j \neq i}a_{ij}^2\tau_{j,k}^2$. Noting that $\sum_{k=0}^\infty \eta_k=\infty$, $\sum_{k=0}^\infty \eta_k^2 <\infty$, $\sum_{k=0}^\infty \frac{\gamma_k^2}{\eta_k}<\infty$, from Assumption \ref{weaktao} and Lemma 3, we have $\lim_{k\to \infty}\Vert y_{i,k}-\Bar{x}_k\Vert ^2 =0$ holds almost surely. 

\subsection{Proof of Theorem 1}

By Assumption \ref{lipschitz} and update rule (\ref{update1b}), there is 
\begin{align}
    \Vert &\lambda_{i,k+1}-\lambda_i^*\Vert ^2 \nonumber\\
    =&\Vert \Pi_{\Lambda_i}[\lambda_{i,k}+\gamma_k g(y_{i,k})]-\lambda_i^*\Vert ^2\nonumber\\
    \leq~& \Vert \lambda_{i,k}-\lambda_i^*\Vert ^2 +2\gamma_k(\lambda_{i,k}-\lambda_i^*)^\mathrm{T}g(y_{i,k}) +\gamma_k^2g^2(y_{i,k}) \nonumber\\
    \leq~& \Vert \lambda_{i,k}-\lambda_i^*\Vert ^2  +2\gamma_k(\lambda_{i,k}-\lambda_i^*)^\mathrm{T}(g(y_{i,k})-g(\Bar{x}_k))\nonumber\\
    &+\gamma_k^2g^2(y_{i,k}) +2\gamma_k(\lambda_{i,k}-\lambda_i^*)^\mathrm{T}g(\Bar{x}_k)\nonumber\\
     \leq~& \Vert \lambda_{i,k}-\lambda_i^*\Vert ^2  +2\gamma_kl_g(\lambda_{i,k}-\lambda_i^*)^\mathrm{T} (y_{i,k}-\Bar{x}_k)\nonumber\\
    &+\gamma_k^2C_g^2 +2\gamma_k(\lambda_{i,k}-\lambda_i^*)^\mathrm{T}g(\Bar{x}_k).
\end{align}

Using the definition $\nabla_{\lambda_i} \mathcal{L}_i(\mathbf{x}_k, \lambda_{i,k})=g(\Bar{x}_k)$ and the relationship $(\lambda_{i,k}-\lambda_i^*)^\mathrm{T}\nabla_{\lambda_i} \mathcal{L}_i(\mathbf{x}_k, \lambda_{i,k})=\mathcal{L}_i(\mathbf{x}_k, \lambda_{i,k})-\mathcal{L}_i(\mathbf{x}_k, \lambda_i^*)$, it holds that
\begin{align}
    \Vert &\lambda_{i,k+1}-\lambda_i^*\Vert ^2 \nonumber\\
    \leq ~& \Vert \lambda_{i,k}-\lambda_i^*\Vert ^2  +2\gamma_kl_g(\lambda_{i,k}-\lambda_i^*)^\mathrm{T}(y_{i,k}-\Bar{x}_k)\nonumber\\
    &+\gamma_k^2C_g^2 +2\gamma_k(\mathcal{L}_i(\mathbf{x}_k, \lambda_{i,k})-\mathcal{L}_i(\mathbf{x}_k, \lambda_i^*)). \label{lam1}
\end{align}

Next, according to Assumption 1 and (\ref{update1a}), we obtain
\begin{align}
    \Vert &x_{i,k+1}-x_i^*\Vert ^2 \nonumber\\
    &=\Vert \Pi_{\Omega_i}[x_{i,k}-\gamma_k(\nabla_i J_i(x_{i,k}, y_{i,k})+\nabla_ig(y_{i,k})\lambda_{i,k})]-x_i^*\Vert ^2\nonumber\\
    & \leq \Vert x_{i,k}-x_i^*\Vert ^2 +\gamma_k^2(l_J+l_g\Bar{\Lambda}_i)^2\nonumber\\
    &~~-2\gamma_k(x_{i,k}-x_i^*)^\mathrm{T}(\nabla_i J_i(x_{i,k}, y_{i,k})+\nabla_ig(y_{i,k})\lambda_{i,k}).  \label{lam2}
 \end{align}

Using the definition $\nabla_{x_i} \mathcal{L}_i(\mathbf{x}_k, \lambda_{i,k})=\nabla_i J_i(x_{i,k}, \Bar{x}_k)+\nabla_ig(\Bar{x}_k)\lambda_{i,k}$ and the relationship $(x_{i,k}-x_i^*)\nabla_{x_i} \mathcal{L}_i(\mathbf{x}_k, \lambda_{i,k})\leq \mathcal{L}_i(\mathbf{x}_k, \lambda_{i,k})-\mathcal{L}_i(\mathbf{x}^*, \lambda_{i,k})$, under Assumption 2, there is
\begin{align}
    &~~~-2\gamma_k(x_{i,k}-x_i^*)^\mathrm{T}(\nabla_i J_i(x_{i,k}, y_{i,k})+\nabla g(y_{i,k})\lambda_{i,k})\nonumber\\
    & \leq -2\gamma_k(x_{i,k}-x_i^*)^\mathrm{T}(\nabla_i J_i(x_{i,k}, \Bar{x}_k)+\nabla g(\Bar{x}_k)\lambda_{i,k})\nonumber\\
    & ~~~-2\gamma_k(x_{i,k}-x_i^*)^\mathrm{T}(\nabla_i J_i(x_{i,k}, y_{i,k})-\nabla_i J_i(x_{i,k}, \Bar{x}_k))\nonumber\\
    & ~~~-2\gamma_k (x_{i,k}-x_i^*)^\mathrm{T}(\nabla g(y_{i,k}) -\nabla g(\Bar{x}_k))\nonumber\\    
    & \leq -2\gamma_k(\mathcal{L}_i(\mathbf{x}_k, \lambda_{i,k})-\mathcal{L}_i(\mathbf{x}^*, \lambda_{i,k}))\nonumber\\
    & ~~~-2\gamma_k (G_J+G_g)(x_{i,k}-x_i^*)^\mathrm{T}(y_{i,k}-\Bar{x}_k).
    \label{lam3}
\end{align}

Combine (\ref{lam1}), (\ref{lam2}) and (\ref{lam3}), we have
\begin{align}
    \Vert &x_{i,k+1}-x_i^*\Vert ^2 + \Vert \lambda_{i,k+1}-\lambda_i^*\Vert ^2 \nonumber\\
     \leq~ &\Vert x_{i,k}-x_i^*\Vert ^2+\Vert \lambda_{i,k}-\lambda_i^*\Vert ^2  +\gamma_k^2(l_J+l_g\Bar{\Lambda}_i)^2 +\gamma_k^2C_g^2\nonumber\\
    & -2\gamma_k (G_J+G_g)(x_{i,k}-x_i^*)^\mathrm{T}(y_{i,k}-\Bar{x}_k)\nonumber\\
    & +2\gamma_kl_g(\lambda_{i,k}-\lambda_i^*)^\mathrm{T}(y_{i,k}-\Bar{x}_k)\nonumber\\
    &-2\gamma_k(\mathcal{L}_i(\mathbf{x}_k, \lambda_i^*)-\mathcal{L}_i(\mathbf{x}^*, \lambda_{i,k})).
\end{align}

According to the saddle point definition, we have $\mathcal{L}_i(\mathbf{x}_k, \lambda_i^*)-\mathcal{L}_i(\mathbf{x}^*, \lambda_{i,k})>0$. Let $a
_k=1$, $u_k=2\gamma_k (G_J+G_g)(x_{i,k}-x_i^*)^\mathrm{T}(y_{i,k}-\Bar{x}_k)+2\gamma_k(\mathcal{L}_i(\mathbf{x}_k, \lambda_i^*)-\mathcal{L}_i(\mathbf{x}^*, \lambda_{i,k}))$ and  $b_k=\gamma_k^2(l_J+l_g\Bar{\Lambda}_i)^2 +\gamma_k^2C_g^2+2\gamma_kl_g(\lambda_{i,k}-\lambda_i^*)^\mathrm{T}(y_{i,k}-\Bar{x}_k)$. Noting that $\sum_{k=0}^\infty \eta_k=\infty$, $\sum_{k=0}^\infty \eta_k^2 <\infty$, $\sum_{k=0}^\infty \frac{\gamma_k^2}{\eta_k}<\infty$, from Lemmas~4--5, we know that $\Vert x_{i,k}-x_i^*\Vert ^2+\Vert \lambda_{i,k}-\lambda_i^*\Vert ^2$ converges almost surely and $\sum_{k=0}^\infty\gamma_k(\mathcal{L}_i(\mathbf{x}_k, \lambda_i^*)-\mathcal{L}_i(\mathbf{x}^*, \lambda_{i,k})) \leq \infty$ holds almost surely. Since $\gamma_k$ is non-summable, we have that the limit inferior of $\mathcal{L}_i(\mathbf{x}_k, \lambda_i^*)-\mathcal{L}_i(\mathbf{x}^*, \lambda_{i,k})$ is zero, i.e., there exists a subsequence $\{l_k\}$ such that $\lim_{k \to \infty} (\mathcal{L}_i(x_{l_k}, \lambda_i^*)-\mathcal{L}_i(\mathbf{x}^*, \lambda_{l_k}))=0$ a.s. This further implies
that $\{(x_{l_k},\lambda_{l_k})\}$ converges to $(x_i^*,\lambda_i^*)$. Then $\lim_{k \to \infty}( \Vert x_{l_k}-x_i^*\Vert ^2+\Vert \lambda_{l_k}-\lambda_i^*\Vert ^2 )=0 $. Since we have shown that $\Vert x_{i,k}-x_i^*\Vert ^2+\Vert \lambda_{i,k}-\lambda_i^*\Vert ^2$ converges almost surely for any saddle point $(x_i^*,\lambda_i^*)$, we have  $\lim_{k \to \infty}( \Vert x_{i,k}-x_i^*\Vert ^2+\Vert \lambda_{i,k}-\lambda_i^*\Vert ^2)=0 $ a.s.

\subsection{Proof of Theorem 2}

  We take step sizes $\eta_k =\frac{1}{k^s}$, $\gamma_k=\frac{1}{k^t}$, where $0.5<s \leq t\leq 1$ and $2t-s>1$, and under Assumptions 1--6, those conditions satisfy the conditions in Lemma \ref{ybarx} and then the convergence result therefore follows immediately from the application of Theorem 1.

 Then the sequences generated under $\mathcal{F}^{(1)}$ and $\mathcal{F}^{(2)}$ are
\begin{align}
y_{i_0,k+1}^{(1)}=y_{i_0,k}^{(1)}+\eta_k\sum_{j\in \mathcal{N}_i}^N a_{ij}(\Tilde{v}_{j,k}^{(1)}-v_{i_0,k}^{(1)})+x_{i_0,k+1}^{(1)}-x_{i_0,k}^{(1)},\nonumber\\
y_{i_0,k+1}^{(2)}=y_{i_0,k}^{(2)}+\eta_k\sum_{j\in \mathcal{N}_i}^N a_{ij}(\Tilde{v}_{j,k}^{(2)}-v_{i_0,k}^{(2)})+x_{i_0,k+1}^{(2)}-x_{i_0,k}^{(2)}.\nonumber
   \end{align}

Define $\Delta x_{i_0,k}= x_{i_0,k}^{(1)}-x_{i_0,k}^{(2)}$ and $\Delta y_{i_0,k}= y_{i_0,k}^{(1)}-y_{i_0,k}^{(2)}$. Then, 
\begin{align}
    &\Delta x_{i_0,k+1}- \Delta x_{i_0,k}\nonumber\\
    &=  \Delta x_{i_0,k} -\gamma_k [\nabla_i J_i(x_{i_0,k}^{(1)},y_{i_0,k}^{(1)})-\nabla_i J_i(x_{i_0,k}^{(2)},y_{i_0,k}^{(2)})]\nonumber\\
    &~~-\gamma_k [\nabla_i g(y_{i_0,k}^{(1)})\lambda_{i_0,k}^{(1)}-\nabla_i g(y_{i_0,k}^{(2)})\lambda_{i_0,k}^{(2)}]- \Delta x_{i_0,k}\nonumber\\
    & \leq 2\gamma_k l_J+2\gamma_k l_g \max \{\Bar{\Lambda}_i\}, 
\end{align}
and
\begin{align}
    \Delta y_{i_0,k+1}&\leq \Delta y_{i_0,k} +\Delta x_{i_0,k+1}-\Delta x_{i_0,k}\nonumber\\
    & \leq \Delta y_{i_0,k} +2\gamma_k l_J +2\gamma_k l_g \max \{\Bar{\Lambda}_i\}.
\end{align}

Since $\Delta y_{i_0,0}=0 $, there is 
\begin{align}
~~~~~~ \Vert \Delta y_{i_0,k}\Vert &\leq \sum_{s=0}^{k-1} (2\gamma_s l_J +2\gamma_s l_g \max \{\Bar{\Lambda}_i\})\nonumber\\
&\leq 2(l_J+l_g \max \{\Bar{\Lambda}_i\})t\ln(k). \label{deltay}
\end{align}

From $\Delta y_{i_0,k} \leq 2\theta$, there exists $l_c > 0$ such that $y_{ij,k}, y_{ij,k}' \in [l_c\theta, (l_c+1)\theta]$. Without loss of generality, we proceed with two cases.

\textbf{Case~1:} If $y_{ij,k}\in [(l_c-1)\theta, l_c\theta]$ and $y_{ij,k}'\in [ l_c\theta,(l_c+1)\theta]$, then
\begin{align}
~~~~\delta_k&=\mathbb{P}[\mathcal{C}(y_{ij,k})=l_c\theta|y_i]-\mathbb{P}[\mathcal{C}(y_{ij,k}')=l_c\theta|y_i']\nonumber\\
     &\leq |y_i/\theta-l_c-(1+l_c-y_i'/\theta) |\nonumber\\
     &= |(y_i+y_i')/\theta-2l_c-1 |\nonumber\\
     &\leq |(l_c\theta+(l_c+1)\theta)/\theta-2l_c-1 |=0.\nonumber
\end{align}

In a similar way, one can obtain the same relationship when $y_{ij,k}\in[ l_c\theta,(l_c+1)\theta] $ and $y_{ij,k}'\in [(l_c-1)\theta, l_c\theta]$.

\textbf{Case~2:} If $y_{ij,k}, y_{ij,k}' \in [(l_c-1)\theta, l_c\theta]$, then
\begin{align}
~~~~\delta_k&=\mathbb{P}[\mathcal{C}(y_{ij,k})=l_c\theta|y_i]-\mathbb{P}[\mathcal{C}(y_{ij,k}')=l_c\theta|y_i']\nonumber\\
     &\leq |1+l_c-y_i/\theta-(1+l_c-y_i'/\theta) |\nonumber\\
     &\leq \frac{|\Delta y_{i_0,k}|}{\theta}. \nonumber
\end{align}

In a similar way, one can obtain the same relationship when $y_{ij,k}, y_{ij,k}'\in [ l_c\theta,(l_c+1)\theta]$.

Also there is $ |\Delta y_{i_0j,k}| \leq 2\theta$ and $ |\Delta y_{i_0,k}| =\sum_{j\in \mathcal{N}_i}^N |\Delta y_{i_0j,k}| \leq 2N\theta$. Additionally,
according to (\ref{deltay}), we have $ |\Delta y_{i_0j,k}| \leq 2\sqrt{N}(l_J+l_g \max \{\Bar{\Lambda}_i\})t\ln(k)$. Moreover, it should be noted that in differential privacy,
$\delta_k$ should be a small parameter in $(0, 1)$. Hence, we derive the expression of $\delta_k$ as
\begin{align}
~~~~~~~~ \delta_k= \min\{ 1, \frac{2t\ln(k)\sqrt{N}}{\theta}(l_J+l_g \max \{\Bar{\Lambda}_i\}) \}. \nonumber
\end{align}

\bibliographystyle{IEEEtran}
\bibliography{document}
\end{document}